\documentclass[10pt,twoside]{article}

\usepackage{amsmath, amsthm, amssymb, mathrsfs, sansmath, setspace}

\newcommand{\ds}	{\displaystyle}

\newtheorem{theorem}{Theorem}
\newtheorem{lemma}{Lemma}
\newtheorem{proposition}{Proposition}
\newtheorem{remark}{Remark}

\title{Uniqueness of Conformal Ricci Flow using Energy Methods}
\author{Thomas Bell}

\begin{document}

\maketitle

\begin{abstract}
We analyze an energy functional associated to Conformal Ricci Flow along closed manifolds with constant negative scalar curvature. Given initial conditions we use this functional to demonstrate the uniqueness of both the metric and the pressure function along Conformal Ricci Flow.
\end{abstract}

\section{Introduction}

The uniqueness of Ricci Flow on closed manifolds was originally proven by Hamilton \cite{Ha1}. Later on, Chen and Zhu proved the uniqueness on complete noncompact manifolds with bounded curvature \cite{CZ}. Both proofs utilize DeTurck Ricci Flow. Recently Kotschwar used energy techniques to give another proof of the uniqueness on complete manifolds \cite{Ko}. Kotchwar's proof does not rely on DeTurck Ricci Flow. A natural question is whether similar techniques can be applied to demonstrate the uniqueness of other geometric flows. One of these flows we have in mind is Conformal Ricci Flow, introduced by Fischer \cite{Fi}. Conformal Ricci Flow is, like Ricci Flow, a weakly parabolic flow of the metric on manifolds. Unlike Ricci Flow, Conformal Ricci Flow is restricted to the class of metrics of constant scalar curvature.

Let $(M^n,g_0)$ be a smooth n-dimensional Riemannian manifold with a metric $g_0$ of constant scalar curvature $s_0$. Conformal Ricci Flow on $M$ is defined as follows:
\begin{equation}\label{CRFequation}
\left\{\begin{tabular}{ccl}$\ds\frac{\partial g}{\partial t}$&$=$&$-2\text{Ric}_{g(t)}+2\frac{s_0}{n}g(t)-2p(t)g(t)$\\$s\bigl(g(t)\bigr)$&$=$&$s_0$\end{tabular}\right. \text{on}~M\times [0,T].
\end{equation}

Here $g(t),~t\in [0,T]$, is a family of metrics on $M$ with $g(0)=g_0$, $s\bigl(g(t)\bigr)$ is the scalar curvature of $g(t)$, and $p(t),~t\in[0,T]$, is a family of functions on $M$. In \cite{Fi} and \cite{LQZ} we see that (\ref{CRFequation}) is equivalent to the following system:
\begin{equation}\label{equationp}
\left\{\begin{tabular}{ccl}$\ds\frac{\partial g}{\partial t}$&$=$&$-2\text{Ric}_{g(t)}+2\frac{s_0}{n}g(t)-2p(t)g(t)$\\$\bigl((n-1)\Delta_{g(t)}+s_0\bigr)p(t)$&$=$&$-\left<\text{Ric}_{g(t)}-\frac{s_0}{n}g(t),\text{Ric}_{g(t)}-\frac{s_0}{n}g(t)\right>$\end{tabular}\right.
\end{equation}

Throughout this paper we will use $V$ to denote the following symmetric $2$-tensor:
\begin{equation}\label{Eq V}
V(t)=\text{Ric}_{g(t)}-\frac{s_0}{n}g(t)+p(t)g(t)
\end{equation}

In this paper we use Kotchwar's idea to give a proof of the uniqueness of Conformal Ricci Flow for closed manifolds with metrics of constant negative scalar curvature. Such uniqueness has been observed by Lu, Qing and Zheng using DeTurck Conformal Ricci Flow \cite{LQZ}. More precisely we will prove the following uniqueness theorem of Conformal Ricci Flow:

\begin{theorem}\label{Thm Main}
Let $(M^n,g_0)$ be a closed manifold with constant negative scalar curvature $s_0$. Suppose $\bigl(g(t),p(t)\bigr)$ and $\bigl(\tilde{g}(t),\tilde{p}(t)\bigr)$ are two solutions of (\ref{CRFequation}) on $M\times[0,T]$ with $\tilde{g}(0)=g(0)$. Then $\bigl(\tilde{g}(t),\tilde{p}(t)\bigr)=\bigl(g(t),p(t)\bigr)$ for $0\leq t\leq T$.
\end{theorem}

\section{The Differences between $g(t)$ and $\tilde{g}(t)$}

Let $g(t)$ and $\tilde{g}(t)$ be as in Theorem \ref{Thm Main}. We will treat $g$ as our background metric and $\tilde{g}$ as our alternative metric. Let $\nabla,\tilde{\nabla}$ be the Riemannian connections of $g$ and $\tilde{g}$ respectively. Similarly, let $R,\tilde{R}$ represent the full Riemannian curvature tensors of $g$ and $\tilde{g}$ respectively.

Let $h=g-\tilde{g}$. Let $A=\nabla-\tilde{\nabla}$. Explicitly, $A^i_{jk}=\Gamma^i_{jk}-\tilde{\Gamma}^i_{jk}$ where $\Gamma^i_{jk}$ and $\tilde{\Gamma}^i_{jk}$ are the Christoffel symbols of $\nabla$ and $\tilde{\nabla}$ respectively. Also let $S=R-\tilde{R},~q=p-\tilde{p}$.

In this section we find bounds on  $h,~A,~S,~q,~\nabla{q}$ and $\nabla\nabla q$ (see Propositions \ref{Prop hAS bounds} and \ref{Prop q bounds}). Throughout this paper we will use the convention $X\ast Y$ to denote any finite sum of tensors of the form $X\cdot Y$. We use C(X) to denote a finite sum of tensors of the form $X$.

\subsection{Preliminary Calculations}

First we calculate some useful expressions for quantities which will arise in the proofs of Propositions \ref{Prop hAS bounds} and \ref{Prop q bounds}. We calculate
\begin{equation*}
g^{ij}-\tilde{g}^{ij}=g^{ik}(\tilde{g}^{j\ell}\tilde{g}_{k\ell})-\tilde{g}^{j\ell}(g^{ik}g_{k\ell})=-g^{ik}\tilde{g}^{j\ell}h_{k\ell},
\end{equation*}

i.e.
\begin{equation*}
g^{-1}-\tilde{g}^{-1}=\tilde{g}^{-1}\ast h.
\end{equation*}

If $X$ is any tensor which is not a function we have
\begin{equation*}
\bigl(\nabla-\tilde{\nabla}\bigr)X=A\ast X.
\end{equation*}

We check this when $X$ is a $(1,1)$-tensor. Calculating in local coordinates we see
\begin{align*}
\bigl(\nabla_i-\tilde{\nabla}_i\bigr)X_j^k&=\partial_iX_j^k-\Gamma_{ij}^{\ell}X_{\ell}^k+\Gamma_{i\ell}^kX_j^{\ell}-\partial_iX_j^k+\tilde{\Gamma}_{ij}^{\ell}X_{\ell}^k-\tilde{\Gamma}_{i\ell}^kX_j^{\ell}\\
&=A_{i\ell}^kX_j^{\ell}-A_{ij}^{\ell}X_{\ell}^k= A\ast X.
\end{align*}

If $f$ is a function however, then we have the following:
\begin{equation*}
\bigl(\nabla_i-\tilde{\nabla}_i\bigr)f=\bigl(g^{ij}-\tilde{g}^{ij}\bigr)\partial_i f=-g^{ik}\tilde{g}^{j\ell}h_{k\ell}\partial_i f=-g^{ik}h_{k\ell}\tilde{\nabla}_{\ell}f,
\end{equation*}
or in other words
\begin{equation*}
\bigl(\nabla-\tilde{\nabla}\bigr)f=h\ast\tilde{\nabla}f.
\end{equation*}

We now calculate
\begin{equation*}
\nabla\tilde{g}^{-1}=\bigl(\nabla-\tilde{\nabla}\bigr)\tilde{g}^{-1}=\tilde{g}^{-1}\ast A.
\end{equation*}

The following calculation will also be important.
\begin{equation*}
\nabla_ih_{jk}=\nabla_ig_{jk}-\nabla_i\tilde{g}_{jk}=-\bigl(\nabla_i-\tilde{\nabla}_i\bigr)\tilde{g}_{jk}.
\end{equation*}

Thus we have
\begin{equation*}
\nabla h=\tilde{g}\ast A.
\end{equation*}

Now we are able to calculate the following for a function $f$.
\begin{align*}
\nabla\bigl(\nabla-\tilde{\nabla}\bigr)f&=\nabla\bigl(h\ast\tilde{\nabla}f\bigr)\\
&=\nabla h\ast\tilde{\nabla}f+h\ast\bigl(\nabla-\tilde{\nabla}\bigr)\tilde{\nabla}f+h\ast\tilde{\nabla}\tilde{\nabla}f\\
&=\tilde{g}\ast A\ast\tilde{\nabla}f+h\ast A\ast\tilde{\nabla}f+h\ast\tilde{\nabla}\tilde{\nabla}f.
\end{align*}

Now let
\begin{align}
U^a_{ijk\ell}&=g^{ab}\nabla_b\tilde{R}_{ijk\ell}-\tilde{g}^{ab}\tilde{\nabla}_b\tilde{R}_{ijk\ell}\label{U1}\\
&=g^{ab}(\nabla_b-\tilde{\nabla}_b)\tilde{R}_{ijk\ell}+\bigl(g^{ab}-\tilde{g}^{ab}\bigr)\tilde{\nabla}_b\tilde{R}_{ijk\ell}\notag\\
&=A\ast\tilde{R}+\tilde{g}^{-1}\ast h\ast\tilde{\nabla}\tilde{R}\notag,
\end{align}
and we may calculate
\begin{align*}
\nabla_a\bigl(g^{ab}\nabla_bR-\tilde{g}^{ab}\tilde{\nabla}_b\tilde{R}\bigr)&=\nabla_a\bigl(g^{ab}\nabla_b\tilde{R}-\tilde{g}^{ab}\tilde{\nabla}_b\tilde{R}\bigr)+g^{ab}\nabla_a\nabla_b\bigl(R-\tilde{R}\bigr)\notag\\
&=\operatorname{div}U+\Delta S.
\end{align*}

We summarize the above calculations in the following Lemma:
\begin{lemma}\label{Lem estimates}
Using the notation defined at the beginning of this section,
\begin{align}
&g^{-1}-\tilde{g}^{-1}=\tilde{g}^{-1}\ast h\\
&\bigl(\nabla-\tilde{\nabla}\bigr)X=A\ast X\\
&\bigl(\nabla-\tilde{\nabla}\bigr)f=h\ast\tilde{\nabla}f\label{hnablaf}\\
&\nabla\tilde{g}^{-1}=\tilde{g}^{-1}\ast A\\
&\nabla h=\tilde{g}\ast A\\
&\nabla\bigl(\nabla-\tilde{\nabla}\bigr)f=\tilde{g}\ast A\ast\tilde{\nabla}f+h\ast A\ast\tilde{\nabla}f+h\ast\tilde{\nabla}\tilde{\nabla}f\label{fnablaf}\\
&U=A\ast\tilde{R}+\tilde{g}^{-1}\ast h\ast\tilde{\nabla}\tilde{R}\label{U2}\\
&\nabla_a\bigl(g^{ab}\nabla_bR-\tilde{g}^{ab}\tilde{\nabla}_b\tilde{R}\bigr)=\operatorname{div}U+\Delta S\label{U3}
\end{align}
where $U$ is defined in (\ref{U1}).
\end{lemma}

\subsection{Bounds on Time Derivatives of $h,~A$ and $S$}

In this subsection we derive bounds on the time derivatives of $h$, $A$ and $S$. In particular we will prove the following proposition. Here, as well as throughout this paper, $C$ will denote a constant dependent only upon $n$ while $N$ will denote a constant with further dependencies.

\begin{proposition}\label{Prop hAS bounds}
Let $\bigl(g(t),p(t)\bigr)$ and $\bigl(\tilde{g}(t),\tilde{p}(t)\bigr)$ be two solutions of (\ref{CRFequation}) on $M\times[0,T]$.
Using the notation defined at the beginning of this section, there exist constants $N_h,~N_A$ and $N_S$ such that
\begin{align}
\left|\frac{\partial}{\partial t}h\right|&\leq N_h|h|+C\bigl(|S|+|q|\bigr)\label{h deriv}\\
\left|\frac{\partial}{\partial t}A\right|&\leq N_A\bigl(|h|+|A|\bigr)+C\bigl(|\nabla S|+|\nabla q|\bigr)\label{A deriv}\\
\left|\frac{\partial}{\partial t}S-\Delta S-\operatorname{div}U\right|&\leq N_S\bigl(|h|+|A|+|S|+|q|\bigr)+C|\nabla\nabla q|\label{S deriv}
\end{align}
where $U$ is defined in (\ref{U1}).
\end{proposition}
\begin{proof}

We start with the time derivative of $h$. By (\ref{CRFequation}) we have
\begin{align*}
\frac{\partial}{\partial t}h_{ij} &=-2(R_{ij}-\tilde{R}_{ij})+2\frac{s_0}{n}(g_{ij}-\tilde{g}_{ij})-2(p\,g_{ij}-\tilde{p}\,\tilde{g}_{ij})\\
&=-2S^k_{kij}+2\frac{s_0}{n}h_{ij}-2\bigl[(p-\tilde{p}){g}_{ij}+\tilde{p}(g_{ij}-\tilde{g}_{ij})\bigr]\\
&=-2S^k_{kij}+2\frac{s_0}{n}h_{ij}-2q\,{g}_{ij}-2\tilde{p}\,h_{ij}.
\end{align*}

Hence
\begin{equation*}
\frac{\partial}{\partial t}h=C(S)+C(s_0h)+C(q)+\tilde{p}\ast h
\end{equation*}
and
\begin{equation}\label{h3}
\left|\frac{\partial}{\partial t}h\right|\leq C\Bigl(\bigl(|s_0|+|\tilde{p}|\bigr)|h|+|S|+|q|\Bigr).
\end{equation}
This proves (\ref{h deriv}).

Recall the definition of $V$ from (\ref{Eq V}):
\begin{equation}
V(t)=\text{Ric}_{g(t)}-\frac{s_0}{n}g(t)+p(t)g(t).
\end{equation}

We may define $\tilde{V}$ similarly using our alternate metric $\tilde{g}$. Since $V$ and $\tilde{V}$ are symmetric $2$-tensors, then by \cite[p. 108]{CLN} we may calculate
\begin{equation}\label{A1}
\frac{\partial}{\partial t}A^k_{ij}=\tilde{g}^{k\ell}\bigl(\tilde{\nabla}_i\tilde{V}_{j\ell}+\tilde{\nabla}_j\tilde{V}_{i\ell}-\tilde{\nabla}_{\ell}\tilde{V}_{ij}\bigr)-g^{k\ell}\bigl(\nabla_iV_{j\ell}+\nabla_jV_{i\ell}-\nabla_{\ell}V_{ij}\bigr).
\end{equation}
We proceed to calculate
\begin{align}
&\tilde{g}^{k\ell}\tilde{\nabla}_i\tilde{V}_{j\ell}-g^{k\ell}\nabla_iV_{j\ell}\notag\\
=&\tilde{g}^{k\ell}(\tilde{\nabla}_i\tilde{R}_{j\ell})-g^{k\ell}(\nabla_iR_{j\ell})+ \tilde{g}^{k\ell}\tilde{\nabla}_i(\tilde{p}~\tilde{g}_{j\ell})-g^{k\ell}\nabla_i(p~g_{j\ell})\notag\\
=&\bigl(\tilde{g}^{k\ell}-g^{k\ell}\bigr)\tilde{\nabla}_i\tilde{R}_{j\ell}+g^{k\ell}(\tilde{\nabla}_i-\nabla_i)\tilde{R}_{j\ell}-g^{k\ell}\nabla_i(S^m_{mj\ell})+\delta^k_j\tilde{\nabla}_i\tilde{p}-\delta^k_j\nabla_ip\notag\\
=&\tilde{g}^{-1}\ast h\ast\tilde{\nabla}\tilde{R}+A\ast\tilde{R}+C(\nabla S)+h\ast\tilde{\nabla}\tilde{p}+C(\nabla q),\label{A2}
\intertext{where we have used (\ref{hnablaf}) to get the last equality. Similarly we find}
&\tilde{g}^{k\ell}\tilde{\nabla}_j\tilde{V}_{i\ell}-g^{k\ell}\nabla_jV_{i\ell}\notag\\
=&\tilde{g}^{-1}\ast h\ast\tilde{\nabla}\tilde{R}+A\ast\tilde{R}+C(\nabla S)+h\ast\tilde{\nabla}\tilde{p}+C(\nabla q).\label{A3}
\end{align}
Now we consider
\begin{align}
&-\tilde{g}^{k\ell}\tilde{\nabla}_{\ell}\tilde{V}_{ij}+g^{k\ell}\nabla_{\ell}V_{ij}\notag\\
=&\tilde{g}^{-1}\ast h\ast\tilde{\nabla}\tilde{R}+A\ast\tilde{R}+C(\nabla S)+\tilde{g}^{k\ell}\tilde{g}_{ij}\tilde{\nabla}_{\ell}\tilde{p}-g^{k\ell}g_{ij}\nabla_{\ell}p\notag\\
=&\tilde{g}^{-1}\ast h\ast\tilde{\nabla}\tilde{R}+A\ast\tilde{R}+C(\nabla S)+\bigl(\tilde{g}^{k\ell}-g^{k\ell}\bigr)\tilde{g}_{ij}\tilde{\nabla}_{\ell}\tilde{p}+g^{k\ell}\bigl(\tilde{g}_{ij}-g_{ij}\bigr)\tilde{\nabla}_{\ell}\tilde{p}\notag\\
&+g^{k\ell}g_{ij}(\tilde{\nabla}_{\ell}-\nabla_{\ell})\tilde{p}+g^{k\ell}g_{ij}\nabla_{\ell}(\tilde{p}-p)\notag\\
=&\tilde{g}^{-1}\ast h\ast\tilde{\nabla}\tilde{R}+A\ast\tilde{R}+C(\nabla S)+\tilde{g}^{-1}\ast h\ast\tilde{g}\ast\tilde{\nabla}\tilde{p}+h\ast\tilde{\nabla}\tilde{p}+C(\nabla q).\label{A4}
\end{align}

Hence by (\ref{A1}), (\ref{A2}), (\ref{A3}) and (\ref{A4}),
\begin{equation*}
\frac{\partial}{\partial t}A=\tilde{g}^{-1}\ast h\ast\tilde{\nabla}\tilde{R}+A\ast\tilde{R}+C(\nabla S)+ h\ast\tilde{\nabla}\tilde{p}+C(\nabla q)+\tilde{g}^{-1}\ast h\ast \tilde{g}\ast\tilde{\nabla}\tilde{p}
\end{equation*}
and
\begin{equation}
\left|\frac{\partial}{\partial t}A\right|\leq C\Bigl(\bigl(|\tilde{g}^{-1}||\tilde{\nabla}\tilde{R}|+|\tilde{\nabla}\tilde{p}|+|\tilde{g}^{-1}||\tilde{g}||\tilde{\nabla}\tilde{p}|\bigr)|h|+|\tilde{R}||A|+|\nabla S|+|\nabla q|\Bigr).
\end{equation}
This proves (\ref{A deriv}).

By \cite[eqn. (2.67)]{CLN} we have
\begin{align}
\frac{\partial}{\partial t}R^{\ell}_{ijk}&=g^{\ell m}\bigl(\nabla_i\nabla_kV_{jm}-\nabla_i\nabla_mV_{jk}-\nabla_j\nabla_kV_{im}+\nabla_j\nabla_mV_{ik}\bigr)\notag\\
&\qquad-g^{\ell m}\bigl(R_{ijk}^rV_{rm}+R_{ijm}^qV_{kq}\bigr)\notag\\
&=g^{\ell m}\bigl(-\nabla_i\nabla_kR_{jm}+\nabla_i\nabla_mR_{jk}+\nabla_j\nabla_kR_{im}-\nabla_j\nabla_mR_{ik}\bigr)\notag\\
&\qquad+g^{\ell m}\bigl(-g_{jm}\nabla_i\nabla_kp+g_{jk}\nabla_i\nabla_mp+g_{im}\nabla_j\nabla_kp-g_{ik}\nabla_j\nabla_mp\bigr)\notag\\
&\qquad+g^{\ell m}\bigl(R_{ijk}^rR_{rm}+R_{ijm}^rR_{kr}\bigr)-\frac{s_0}{n}g^{\ell m}\bigl(R_{ijk}^rg_{rm}+R_{ijm}^rg_{kr}\bigr)p\notag\\
&\qquad+g^{\ell m}\bigl(R_{ijk}^rg_{rm}+R_{ijm}^rg_{kr}\bigr)p.\label{R1}
\end{align}
Following the calculations in \cite[p. 119-120]{CLN} we have
\begin{align}
\Delta R^{\ell}_{ijk}&=g^{ab}\nabla_a\nabla_bR^{\ell}_{ijk}=g^{ab}\bigl(-\nabla_a\nabla_iR^{\ell}_{jbk}-\nabla_a\nabla_jR^{\ell}_{bik}\bigr)\notag\\
&=g^{ab}\bigl(-\nabla_i\nabla_aR^{\ell}_{jbk}+R^m_{aij}R^{\ell}_{mbk}+R^m_{aib}R^{\ell}_{jmk}+R^m_{aik}R^{\ell}_{jbm}-R^{\ell}_{aim}R^m_{jbk}\notag\\
&\qquad -\nabla_j\nabla_aR^{\ell}_{bik}+R^m_{ajb}R^{\ell}_{mik}+R^m_{aji}R^{\ell}_{bmk}+R^m_{ajk}R^{\ell}_{bim}-R^{\ell}_{ajm}R^m_{bik}\bigr)\notag\\
&=g^{\ell m}\bigl(-\nabla_i\nabla_kR_{jm}+\nabla_i\nabla_mR_{jk}+\nabla_j\nabla_kR_{im}-\nabla_j\nabla_mR_{ik}\bigr)\notag\\
&\qquad+g^{mr}\bigl(-R_{ir}R^{\ell}_{jmk}-R_{jr}R^{\ell}_{mik})\notag\\
&\qquad+g^{ab}\bigl(R^m_{aij}R^{\ell}_{mbk}+R^m_{aik}R^{\ell}_{jbm}-R^{\ell}_{aim}R^m_{jbk}\notag\\
&\qquad\qquad+R^m_{aji}R^{\ell}_{bmk}+R^m_{ajk}R^{\ell}_{bim}-R^{\ell}_{ajm}R^m_{bik}\bigr).\label{R2}
\end{align}
Combining (\ref{R1}) and (\ref{R2}) we have
\begin{align}
\frac{\partial}{\partial t}R^{\ell}_{ijk}&=\Delta R^{\ell}_{ijk}+g^{mr}\bigl(R_{ir}R^{\ell}_{jmk}+R_{jr}R^{\ell}_{mik}\bigr)\notag\\
&+g^{ab}\bigl(-R^m_{aij}R^{\ell}_{mbk}-R^m_{aik}R^{\ell}_{jbm}+R^{\ell}_{aim}R^m_{jbk}\notag\\
&\qquad-R^m_{aji}R^{\ell}_{bmk}-R^m_{ajk}R^{\ell}_{bim}+R^{\ell}_{ajm}R^m_{bik}\bigr)\notag\\
&+g^{\ell m}\bigl(-g_{jm}\nabla_i\nabla_kp+g_{jk}\nabla_i\nabla_mp+g_{im}\nabla_j\nabla_kp-g_{ik}\nabla_j\nabla_mp\bigr)\notag\\
&+g^{\ell m}\bigl(R_{ijk}^rR_{rm}+R_{ijm}^rR_{kr}\bigr)-\frac{s_0}{n}g^{\ell m}\bigl(R_{ijk}^rg_{rm}+R_{ijm}^rg_{kr}\bigr)\notag\\
&+g^{\ell m}\bigl(R_{ijk}^rg_{rm}+R_{ijm}^rg_{kr}\bigr)p.\label{R3}
\end{align}
Hence the evolution of $S$ is
\begin{align}
\frac{\partial}{\partial t}S^{\ell}_{ijk}&=\Delta R^{\ell}_{ijk}-\tilde{\Delta}\tilde{R}^{\ell}_{ijk}\notag\\
&+g^{mr}\bigl(R_{ir}R^{\ell}_{jmk}+R_{jr}R^{\ell}_{jmk}\bigr)-\tilde{g}^{mr}\bigl(\tilde{R}_{ir}\tilde{R}^{\ell}_{jmk}+\tilde{R}_{jr}\tilde{R}^{\ell}_{mik}\bigr)\notag\\
&+g^{ab}\bigl(-R^m_{aij}R^{\ell}_{mbk}-R^m_{aik}R^{\ell}_{jbm}+R^{\ell}_{aim}R^m_{jbk}\notag\\
&\qquad-R^m_{aji}R^{\ell}_{bmk}-R^m_{ajk}R^{\ell}_{bim}+R^{\ell}_{ajm}R^m_{bik}\bigr)\notag\\
&-\tilde{g}^{ab}\bigl(-\tilde{R}^m_{aij}\tilde{R}^{\ell}_{mbk}-\tilde{R}^m_{aik}\tilde{R}^{\ell}_{jbm}+\tilde{R}^{\ell}_{aim}\tilde{R}^m_{jbk}\notag\\
&\qquad-\tilde{R}^m_{aji}\tilde{R}^{\ell}_{bmk}-\tilde{R}^m_{ajk}\tilde{R}^{\ell}_{bim}+\tilde{R}^{\ell}_{ajm}\tilde{R}^m_{bik}\bigr)\notag\\
&+g^{\ell m}\bigl(-g_{jm}\nabla_i\nabla_kp+g_{jk}\nabla_i\nabla_mp+g_{im}\nabla_j\nabla_kp-g_{ik}\nabla_j\nabla_mp\bigr)\notag\\
&-\tilde{g}^{\ell m}\bigl(-\tilde{g}_{jm}\tilde{\nabla}_i\tilde{\nabla}_k\tilde{p}+\tilde{g}_{jk}\tilde{\nabla}_i\tilde{\nabla}_m\tilde{p}+\tilde{g}_{im}\tilde{\nabla}_j\tilde{\nabla}_k\tilde{p}-\tilde{g}_{ik}\tilde{\nabla}_j\tilde{\nabla}_m\tilde{p}\bigr)\notag\\
&+g^{\ell m}\bigl(R_{ijk}^rR_{rm}+R_{ijm}^rR_{kr}\bigr)-\tilde{g}^{\ell m}\bigl(\tilde{R}_{ijk}^r\tilde{R}_{rm}+\tilde{R}_{ijm}^r\tilde{R}_{kr}\bigr)\notag\\
&-\frac{s_0}{n}g^{\ell m}\bigl(R_{ijk}^rg_{rm}+R_{ijm}^rg_{kr}\bigr)+\frac{s_0}{n}\tilde{g}^{\ell m}\bigl(\tilde{R}_{ijk}^r\tilde{g}_{rm}+\tilde{R}_{ijm}^r\tilde{g}_{kr}\bigr)\notag\\
&+g^{\ell m}\bigl(R_{ijk}^rg_{rm}+R_{ijm}^rg_{kr}\bigr)p-\tilde{g}^{\ell m}\bigl(\tilde{R}_{ijk}^r\tilde{g}_{rm}+\tilde{R}_{ijm}^r\tilde{g}_{kr}\bigr)\tilde{p}.\label{R4}
\end{align}
Looking at the individual components, we see
\begin{align}
&\Delta R-\tilde{\Delta}\tilde{R}\notag\\
=&g^{ab}\nabla_a\nabla_bR-\tilde{g}^{ab}\tilde{\nabla}_a\tilde{\nabla}_b\tilde{R}\notag\\
=&\nabla_a(g^{ab}\nabla_bR)-\nabla_a(\tilde{g}^{ab}\tilde{\nabla}_b\tilde{R})+(\nabla_a-\tilde{\nabla}_a)(\tilde{g}^{ab}\tilde{\nabla}_b\tilde{R})\notag\\
=&\nabla_a\bigl(g^{ab}\nabla_bR-\tilde{g}^{ab}\tilde{\nabla}_b\tilde{R}\bigr)+\tilde{g}^{-1}\ast A\ast\tilde{\nabla}\tilde{R},\label{R5}
\end{align}
while
\begin{align}
&g^{-1}RR-\tilde{g}^{-1}\tilde{R}\tilde{R}\notag\\
=&(g^{-1}-\tilde{g}^{-1})(\tilde{R}\tilde{R})+g^{-1}(RR-\tilde{R}\tilde{R})\notag\\
=&\tilde{g}^{-1}\ast h\ast\tilde{R}\ast\tilde{R}+g^{-1}(R-\tilde{R})\tilde{R}+g^{-1}(RR-R\tilde{R})\notag\\
=&\tilde{g}^{-1}\ast h\ast\tilde{R}\ast\tilde{R}+S\ast\tilde{R}+S\ast R,\label{R6}
\end{align}
and
\begin{align}
&g^{-1}g\nabla\nabla p-\tilde{g}^{-1}\tilde{g}\tilde{\nabla}\tilde{\nabla}\tilde{p}\notag\\
=&(g^{-1}-\tilde{g}^{-1})\tilde{g}\tilde{\nabla}\tilde{\nabla}\tilde{p}+g^{-1}(g-\tilde{g})\tilde{\nabla}\tilde{\nabla}\tilde{p}+g^{-1}g(\nabla\nabla p-\tilde{\nabla}\tilde{\nabla}\tilde{p})\notag\\
=&\tilde{g}^{-1}\ast h\ast\tilde{g}\ast\tilde{\nabla}\tilde{\nabla}\tilde{p}+h\ast\tilde{\nabla}\tilde{\nabla}\tilde{p}+g^{-1}g(\nabla-\tilde{\nabla})(\tilde{\nabla}\tilde{p})+g^{-1}g(\nabla\nabla p-\nabla\tilde{\nabla}\tilde{p})\notag\\
=&\tilde{g}^{-1}\ast h\ast\tilde{g}\ast\tilde{\nabla}\tilde{\nabla}\tilde{p}+h\ast\tilde{\nabla}\tilde{\nabla}\tilde{p}+A\ast\tilde{\nabla}\tilde{p}+g^{-1}g\nabla(\nabla-\tilde{\nabla})\tilde{p}+g^{-1}g\nabla\nabla(p-\tilde{p})\notag\\
=&\tilde{g}^{-1}\ast h\ast\tilde{g}\ast\tilde{\nabla}\tilde{\nabla}\tilde{p}+h\ast\tilde{\nabla}\tilde{\nabla}\tilde{p}+A\ast\tilde{\nabla}\tilde{p}+h\ast A\ast\tilde{\nabla}\tilde{p}+C(\nabla\nabla q),\label{R7}
\end{align}
where in the last equality we used (\ref{fnablaf}). We also have
\begin{align}
&g^{-1}gR-\tilde{g}^{-1}\tilde{g}\tilde{R}\notag\\
=&(g^{-1}-\tilde{g}^{-1})\tilde{g}\tilde{R}+g^{-1}(g-\tilde{g})\tilde{R}+g^{-1}g(R-\tilde{R})\notag\\
=&\tilde{g}^{-1}\ast h\ast\tilde{g}\ast\tilde{R}+h\ast\tilde{R}+C(S),\label{R8}
\end{align}
and lastly
\begin{align}
&g^{-1}gRp-\tilde{g}^{-1}\tilde{g}\tilde{R}\tilde{p}\notag\\
=&(g^{-1}-\tilde{g}^{-1})\tilde{g}\tilde{R}\tilde{p}+g^{-1}(g-\tilde{g})\tilde{R}\tilde{p}+g^{-1}g(R-\tilde{R})\tilde{p}+g^{-1}gR(p-\tilde{p})\notag\\
=&\tilde{g}^{-1}\ast h\ast\tilde{g}\ast\tilde{R}\ast\tilde{p}+h\ast\tilde{R}\ast\tilde{p}+S\ast\tilde{p}+R\ast q.\label{R9}
\end{align}

Now by (\ref{R4}), (\ref{R5}), (\ref{R6}), (\ref{R7}), (\ref{R8}) and (\ref{R9}) we see
\begin{align*}
\frac{\partial}{\partial t}S&=\nabla_a\bigl(g^{ab}\nabla_bR-\tilde{g}^{ab}\tilde{\nabla}_b\tilde{R}\bigr)+\tilde{g}^{-1}\ast A\ast\tilde{\nabla}\tilde{R}+\tilde{g}^{-1}\ast h\ast\tilde{R}\ast\tilde{R}\\
&\qquad +S\ast\tilde{R}+S\ast R+\tilde{g}^{-1}\ast h\ast\tilde{g}\ast\tilde{\nabla}\tilde{\nabla}\tilde{p}+h\ast\tilde{\nabla}\tilde{\nabla}\tilde{p}+A\ast\tilde{\nabla}\tilde{p}\\
&\qquad+h\ast A\ast\tilde{\nabla}\tilde{p}+C(\nabla\nabla q)+\tilde{g}^{-1}\ast h\ast\tilde{g}\ast\tilde{R}+h\ast\tilde{R}+C(S)\\
&\qquad+\tilde{g}^{-1}\ast h\ast\tilde{g}\ast\tilde{R}\ast\tilde{p}+h\ast\tilde{R}\ast\tilde{p}+S\ast\tilde{p}+R\ast q.
\end{align*}
Hence by (\ref{U3}) we have
\begin{align}
&\left|\frac{\partial}{\partial t}S-\Delta S-\text{div }U\right|\notag\\
\leq &C\biggl(\Bigl(|\tilde{g}^{-1}||\tilde{R}|^2+|\tilde{g}^{-1}||\tilde{g}||\tilde{\nabla}\tilde{\nabla}\tilde{p}|+|\tilde{\nabla}\tilde{\nabla}\tilde{p}|\notag\\
&\qquad+|\tilde{g}^{-1}||\tilde{g}||\tilde{R}|+|\tilde{R}|+|\tilde{g}^{-1}||\tilde{g}||\tilde{R}||\tilde{p}|+|\tilde{R}||\tilde{p}|\Bigr)|h|\notag\\
&\qquad+\Bigl(|\tilde{g}^{-1}||\tilde{\nabla}\tilde{R}|+|\tilde{\nabla}\tilde{p}|+|h||\tilde{\nabla}\tilde{p}|\Bigr)|A|\notag\\
&\qquad+\Bigl(|\tilde{R}|+|R|+1+|\tilde{p}|\Bigr)|S|+|R||q|+|\nabla\nabla q|\biggr).\label{S2}
\end{align}
This proves (\ref{S deriv}).
\end{proof}
\begin{remark}
Upon closer observation we notice the following dependencies:
\begin{align*}
&N_h=N_h\bigl(n, s_0,|\tilde{p}|\bigr),\\
&N_A=N_A\bigl(n, s_0,|\tilde{g}|,|\tilde{g}^{-1}|,|\tilde{R}|,|\tilde{\nabla}\tilde{R}|,|\tilde{\nabla}\tilde{p}|\bigr),\\
&N_S=N_S\bigl(n, s_0,|\tilde{g}|,|\tilde{g}^{-1}|,|h|,|R|,|\tilde{R}|,|\tilde{\nabla}\tilde{R}|,|\tilde{p}|,|\tilde{\nabla}\tilde{p}|,|\tilde{\nabla}\tilde{\nabla}\tilde{p}|\bigr).
\end{align*}
$M$ is closed, so $M\times[0,T]$ is compact. Thus, given two metrics $g$ and $\tilde{g}$, all of these quantities will be bounded.
\end{remark}

\subsection{Bounds on $q$ and its Spacial Derivatives}

We turn our attention now to finding bounds on the differences between our pressure functions $p$ and $\tilde{p}$. We have the following proposition:

\begin{proposition}\label{Prop q bounds}
Let $\bigl(g(t),p(t)\bigr)$ and $\bigl(\tilde{g}(t),\tilde{p}(t)\bigr)$ be two solutions of (\ref{CRFequation}) on $M\times[0,T]$.
Then there exist constants $N_q$ and $\hat{N}_q$ such that
\begin{align}
\int_M|q|^2d\mu&\leq N_q\int_M\bigl(|h|^2+|A|^2+|S|^2\bigr)d\mu\label{q}\\
\int_M|\nabla q|^2d\mu&\leq N_q\int_M\bigl(|h|^2+|A|^2+|S|^2\bigr)d\mu\label{qq}\\
\int_M|\nabla\nabla q|^2d\mu&\leq \hat{N}_q\int_M\bigl(|h|^2+|A|^2+|S|^2\bigr)d\mu\label{qqq}
\end{align}
\end{proposition}

\begin{proof}

We let $f$ represent any smooth function or tensor. In particular we will let $f$ be represented by the function $q$, the difference of the pressure functions. Since $M$ is compact we have
\begin{align*}
&\int_M\bigl((n-1)\Delta+s_0\bigr)(f)\cdot f~d\mu\\
=&s_0\int_M|f|^2d\mu-(n-1)\int_M\bigl<\nabla f,\nabla f\bigr>d\mu.
\end{align*}
Since $s_0<0$, taking the absolute value gives 
\begin{equation}\label{fbound}
\left|\int_M\bigl((n-1)\Delta+s_0)(f)\cdot fd\mu\right|= |s_0|\int_M|f|^2d\mu+(n-1)\int_M|\nabla f|^2d\mu
\end{equation}
 
Now we deal specifically with $p,~\tilde{p}$ and $q$. By (\ref{equationp}) we have the following equations for the pressure functions $p$ and $\tilde{p}$:
\begin{equation}\label{q1}
\bigl((n-1)\Delta+s_0\bigr)p=-\left<\text{Ric}-\frac{s_0}{n}g,\text{Ric}-\frac{s_0}{n}g\right>
\end{equation}
\begin{equation}\label{q2}
\bigl((n-1)\tilde{\Delta}+s_0\bigr)\tilde{p}=-\left<\tilde{\text{Ric}}-\frac{s_0}{n}\tilde{g},\tilde{\text{Ric}}-\frac{s_0}{n}\tilde{g}\right>.
\end{equation}
Now we calculate
\begin{align}
\Delta p-\tilde{\Delta}\tilde{p}&=g^{ab}\nabla_a\nabla_b p-\tilde{g}^{ab}\tilde{\nabla}_a\tilde{\nabla}_b\tilde{p}\notag\\
&=(g^{-1}-\tilde{g}^{-1})\tilde{\nabla}\tilde{\nabla}\tilde{p}+g^{-1}(\nabla-\tilde{\nabla})\tilde{\nabla}\tilde{p}+g^{-1}\nabla(\nabla-\tilde{\nabla})\tilde{p}+\Delta(p-\tilde{p})\notag\\
&=\tilde{g}^{-1}\ast h\ast\tilde{\nabla}\tilde{\nabla}\tilde{p}+A\ast\tilde{\nabla}\tilde{p}+h\ast A\ast\tilde{\nabla}\tilde{p}+\Delta q.\label{q3}
\end{align}
We also compute
\begin{align}
&-\left<\text{Ric}-\frac{s_0}{n}g,\text{Ric}-\frac{s_0}{n}g\right>+\left<\tilde{\text{Ric}}-\frac{s_0}{n}\tilde{g},\tilde{\text{Ric}}-\frac{s_0}{n}\tilde{g}\right>\notag\\
=&-\bigl(g^{ik}g^{j\ell}R_{ij}R_{k\ell}-\tilde{g}^{ik}\tilde{g}^{j\ell}\tilde{R}_{ij}\tilde{R}_{k\ell}\bigr)+2\frac{s_0}{n}\bigl(g^{ij}R_{ij}-\tilde{g}^{ij}\tilde{R}_{ij}\bigr)\notag\\
=&-(g^{-1}-\tilde{g}^{-1})\tilde{g}^{-1}\tilde{R}\tilde{R}-g^{-1}(g^{-1}-\tilde{g}^{-1})\tilde{R}\tilde{R}-g^{-1}g^{-1}(R-\tilde{R})\tilde{R}\notag\\
&\qquad-g^{-1}g^{-1}R(R-\tilde{R})+2\frac{s_0}{n}(g^{-1}-\tilde{g}^{-1})\tilde{R}+2\frac{s_0}{n}g^{-1}(R-\tilde{R})\notag\\
=&\tilde{g}^{-1}\ast\tilde{g}^{-1}\ast h\ast\tilde{R}\ast\tilde{R}+\tilde{g}^{-1}\ast h\ast\tilde{R}\ast\tilde{R}\notag\\
&\qquad+S\ast\tilde{R}+S\ast R+\tilde{g}^{-1}\ast h\ast\tilde{R} + C(S).\label{q4}
\end{align}

Combining (\ref{q1}), (\ref{q2}), (\ref{q3}) and (\ref{q4}), we see that $q$ satisfies the following Elliptic equation at each time $t\in[0,T]$:
\begin{align}
Lq&=\bigl((n-1)\Delta+s_0\bigr)(q)\notag\\
&=\tilde{g}^{-1}\ast h\ast\tilde{\nabla}\tilde{\nabla}\tilde{p}+A\ast\tilde{\nabla}\tilde{p}+h\ast A\ast\tilde{\nabla}\tilde{p}+\tilde{g}^{-1}\ast\tilde{g}^{-1}\ast h\ast\tilde{R}\ast\tilde{R}\notag\\
&\qquad+\tilde{g}^{-1}\ast h\ast\tilde{R}\ast\tilde{R}+S\ast\tilde{R}+S\ast R+\tilde{g}^{-1}\ast h\ast\tilde{R} + C(S)\label{Lqfunction}
\end{align}
Hence
\begin{equation}\label{qbound1}
|Lq|=\bigl|\bigl((n-1)\Delta+s_0\bigr)(q)\bigr|\leq N\bigl(|h|+|A|+|S|\bigr).
\end{equation}

To find estimates for $q$ and $\nabla q$, we combine (\ref{fbound}) and  (\ref{qbound1}):
\begin{align*}
&|s_0|\int_M|q|^2d\mu+(n-1)\int_M|\nabla q|^2d\mu\\
=&\left|\int_M\bigl((n-1)\Delta+s_0\bigr)(q)\cdot q~d\mu\right|\\
\leq&\int_MN\bigl(|h|+|A|+|S|\bigr)|q|~d\mu\\
\leq&\frac{|s_0|}{2}\int_M|q|^2d\mu+N\int_M\bigl(|h|^2+|A|^2+|S|^2\bigr)d\mu.
\end{align*}
Thus
\begin{equation*}
\frac{|s_0|}{2}\int_M|q|^2d\mu+(n-1)\int_M\bigl|\nabla q|^2d\mu\leq N\int_M\bigl(|h|^2+|A|^2+|S|^2\bigr)d\mu,
\end{equation*}
and we proved (\ref{q}) and (\ref{qq}).

To find an appropriate bound for $|\nabla\nabla q|$ we must turn to Interior Regularity Theory for Elliptic PDE. From (\ref{Lqfunction}) we see that $Lq=f$ is an Elliptic Equation. We then have the following estimate from \cite[p. 229]{Ra}.
\begin{equation*}
|q|_{H^2(W)}\leq K\bigl(|Lq|_{L^2(M)}+|q|_{H^1(M)}\bigr),
\end{equation*}
where $W$ is any compactly supported open subset of $M$ and $K$ depends only upon the coefficients of the operator $L$, the subset $W$ and the manifold $M$. Since $M$ is a closed manifold we may in fact choose $W=M$. Thus we have
\begin{equation}\label{IntEstimate2}
|q|_{H^2(M)}\leq K\bigl(|Lq|_{L^2(M)}+|q|_{H^1(M)}\bigr).
\end{equation}
Upon squaring both sides we observe
\begin{equation}\label{IntEstimate3}
\int_M|\nabla\nabla q|^2d\mu\leq|q|^2_{H^2(M)}\leq K^2\left(\int_M|Lq|^2d\mu+|q|_{H^1(M)}^2\right).
\end{equation}
Now (\ref{q}) and (\ref{qq}) imply that
\begin{equation}\label{IntEstimate4}
|q|^2_{H^1(M)}\leq N\int_M\bigl(|h|^2+|A|^2+|S|^2\bigr)d\mu.
\end{equation}
Combining (\ref{qbound1}), (\ref{IntEstimate3}) and (\ref{IntEstimate4}) we have
\begin{equation*}
\int_M|\nabla\nabla q|^2d\mu\leq N\int_M\bigl(|h|^2+|A|^2+|S|^2\bigr)d\mu,
\end{equation*}
and we proved (\ref{qqq}).

\end{proof}

\begin{remark}
We observe the following dependencies:
\begin{align*}
N_q&=N_q\bigl(n,s_0,|\tilde{g}^{-1}|,|h|,|R|,|\tilde{R}|,|\tilde{\nabla}\tilde{p}|,|\tilde{\nabla}\tilde{\nabla}\tilde{p}|\bigr)\\
\hat{N}_q&=\hat{N}_q\bigl(n,s_0,|\tilde{g}^{-1}|,|h|,|R|,|\tilde{R}|,|\tilde{\nabla}\tilde{p}|,|\tilde{\nabla}\tilde{\nabla}\tilde{p}|,K\bigr)\end{align*}
where $K$ is from (\ref{IntEstimate2}).
\end{remark}

\section{Energy Estimates}

Now we shall approximate the energy
\begin{equation}
\mathcal{E}(t)=\int_M\bigl(|h|^2+|A|^2+|S|^2\bigr)d\mu.
\end{equation}

We also define the following:
\begin{align}
\mathcal{H}(t) &=\int_M|h|^2d\mu\label{H}\\
\mathcal{A}(t) &=\int_M|A|^2d\mu\label{I}\\
\mathcal{S}(t) &=\int_M|S|^2d\mu\label{G}\\
\mathcal{D}(t) &=\int_M|\nabla S|^2d\mu\label{J}
\end{align}

Note that $\mathcal{E}(t)=\mathcal{H}(t)+\mathcal{A}(t)+\mathcal{S}(t)$. We now estimate the evolution of the energy functional under Conformal Ricci Flow, $\mathcal{E}'(t)$, by first estimating the evolutions of $\mathcal{H},~\mathcal{A}$ and $\mathcal{S}$.

\subsection{Evolution of $\mathcal{H}(t)$}

In \cite{LQZ}, Lu, Qing and Zheng give the evolution of the volume element under Conformal Ricci Flow:
\begin{equation}\label{VolEvol}
\frac{\partial}{\partial t}d\mu_{g(t)}=-n p(t)d\mu_{g(t)}
\end{equation}
Hence by (\ref{h deriv}) and (\ref{H}) we have
\begin{align*}
\mathcal{H}'(t)&\leq N\int_M|h|^2d\mu+\int_M2\left<\frac{\partial h}{\partial t},h\right>d\mu\\
&\leq N\mathcal{H}(t)+\int_M2|h|\biggl|\frac{\partial h}{\partial t}\biggr|d\mu\\
&\leq N\mathcal{H}(t)+N\int_M\bigl(|S||h|+|h|^2+|q||h|\bigr)d\mu.
\end{align*}
Now we know that $N\bigl(|S||h|+|q||h|\bigr)\leq N\bigl(|h|^2+|S|^2+|q|^2\bigr)$. Hence
\begin{align}
\mathcal{H}'(t)&\leq N\mathcal{H}(t)+N\int_M\bigl(|S|^2+|q|^2\bigr)d\mu\notag\\
&\leq N\mathcal{H}(t)+N\int_M\bigl(|S|^2+|h|^2+|A|^2\bigr)d\mu\notag\\
&\leq N\mathcal{H}(t)+N\mathcal{S}(t)+N\mathcal{A}(t)=N\mathcal{E}(t).\label{equationh}
\end{align}
 
\subsection{Evolution of $\mathcal{A}(t)$}

By (\ref{A deriv}), (\ref{I}) and (\ref{VolEvol}) we have
\begin{align*}
\mathcal{A}'(t)&\leq N\mathcal{A}(t)+\int_M2|A|\biggl|\frac{\partial A}{\partial t}\biggr|d\mu\\
&\leq N\mathcal{A}(t)+\int_M\Bigl(N|h||A|+N|A|^2+C|\nabla S||A|+C|\nabla q||A|\Bigr)d\mu.
\end{align*}
Now
\begin{equation*}
N|h||A|+C|\nabla S||A|+C|\nabla q||A|\leq N|h|^2+N|A|^2+|\nabla S|^2+|	\nabla q|^2.
\end{equation*}
hence we have that
\begin{align}
\mathcal{A}'(t)&\leq N\mathcal{A}(t)+\int_M\Bigl(N|h|^2+N|A|^2+|\nabla S|^2+|\nabla q|^2\Bigr)d\mu\notag\\
&\leq N\mathcal{A}(t)+N\mathcal{H}(t)+\mathcal{D}(t)+N\int_M\bigl(|h|^2+|A|^2+|S|^2\bigr)d\mu\notag\\
&\leq N\mathcal{A}(t)+N\mathcal{H}(t)+N\mathcal{S}(t)+\mathcal{D}(t)= N\mathcal{E}(t)+\mathcal{D}(t).\label{equationi}
\end{align}

\subsection{Evolution of $\mathcal{S}(t)$}

By (\ref{S deriv}), (\ref{G}) and (\ref{VolEvol}) we have
\begin{align*}
\mathcal{S}'(t)&\leq N\int_M|S|^2d\mu+\int_M2\biggl<\frac{\partial S}{\partial t},S\biggr>d\mu\\
&\leq N\mathcal{S}(t)+\int_M\Bigl(2\bigl<\Delta S+\text{div }V,S\bigr>\\
&\qquad+N\bigl(|h|+|A|+|S|+|q|\bigr)|S|+C|\nabla\nabla q||S|\Bigr)d\mu\\
&\leq N\mathcal{S}(t)+\int_M\Bigl(2\bigl<\Delta S+\text{div }V,S\bigr>\\
&\qquad+N\bigl(|h|^2+|A|^2+|S|^2+|q|^2+|\nabla\nabla q|^2\bigr)\Bigr)d\mu.
\end{align*}
Now by (\ref{q}) and (\ref{qqq}) we have
\begin{align*}
\mathcal{S}'(t)&\leq N\mathcal{S}(t)+N\mathcal{H}(t)+N\mathcal{A}(t)\\
&\qquad+\int_M\biggl(2\bigl<\Delta S+\text{div }V,S \bigr>+N\bigl(|A|^2+|S|^2+|h|^2\bigr)\biggr)d\mu\\
&\leq N\mathcal{S}(t)+N\mathcal{H}(t)+N\mathcal{A}(t)+\int_M2\bigl<\Delta S+\text{div }V,S \bigr>d\mu.
\end{align*}
Upon integrating by parts we get
\begin{align*}
\mathcal{S}'(t)&\leq N\mathcal{E}(t)-2\int_M\bigl<\nabla S+V,\nabla S\bigr>d\mu\notag\\
&\leq N\mathcal{E}(t)-2\int_M|\nabla S|^2d\mu+\int_M2|V||\nabla S|d\mu.
\end{align*}
Now we know that
\begin{equation*}
2|V||\nabla S|\leq |\nabla S|^2+|V|^2\leq|\nabla S|^2+N\bigl(|h|^2+|A|^2\bigr),
\end{equation*}
hence
\begin{equation}\label{equationg}
\mathcal{S}'(t)\leq N\mathcal{E}(t)+N\int_M\bigl(|h|^2+|A|^2\bigr)d\mu-\int_M|\nabla S|^2d\mu\leq N\mathcal{E}(t)-\mathcal{D}(t).
\end{equation}

\subsection{Proof of Main Theorem}

Now we are ready to prove Theorem 1:
\begin{proof}
By (\ref{equationg}), (\ref{equationh}) and (\ref{equationi}) we know that
\begin{equation*}
\mathcal{H}'(t)\leq N\mathcal{E}(t),~\mathcal{A}'(t)\leq N\mathcal{E}(t)+\mathcal{D}(t),~\mathcal{S}'(t)\leq N\mathcal{E}(t)-\mathcal{D}(t),
\end{equation*}
so
\begin{equation*}
\mathcal{E}'(t)\leq N\mathcal{E}(t).
\end{equation*}

Our initial condition $\tilde{g}(0)=g(0)$ tells us that at $t=0$ we have $|h|=|A|=|S|=0$. Therefore by the smoothness and integrability of our solutions we know
\begin{equation*}
\lim_{t\rightarrow 0^+}\mathcal{E}(t)=0,
\end{equation*}
so by Gronwall's Inequality we know that $\mathcal{E}\equiv 0$ on $[0,T]$. Thus for $t\in [0,T]$ we have that $h\equiv 0$ and $g(t)\equiv\tilde{g}(t)$. Also, $\mathcal{E}\equiv 0$ implies $A\equiv 0$ and $S\equiv 0$, so (\ref{q}) forces $q\equiv 0$. Thus $p(t)\equiv\tilde{p}(t)$. Therefore $\bigl(\tilde{g}(t),\tilde{p}(t)\bigr)=\bigl(g(t),p(t)\bigr),~t\in[0,T]$.\\
\end{proof}


\begin{thebibliography}{}

\bibitem{CLN} Chow, Bennett; Lu, Peng; Ni, Lei, \emph{Hamilton's Ricci Flow}. Science Press, Beijing and AMS, (2006)

\bibitem{CZ} Chen, Bing-Long; Zhu, Xi-Ping, \emph{Uniqueness of the Ricci flow on complete noncompact manifolds}. J. Differential Geom. \textbf{74} (2006), 119--154.

\bibitem {Fi} Fischer, Arthur, \emph{An introduction to conformal Ricci flow}. Classical Quantum Gravity \textbf{21} (2004), S171-–S218

\bibitem{Ha1} Hamilton, Richard, \emph{Three-manifolds with positive Ricci curvature}. J. Differential Geom. \textbf{17} (1982), 255–-306

\bibitem{Ko} Kotschwar, Brett, \emph{An energy approach to the problem of uniqueness for the Ricci flow}. ArXiv:1206.3225

\bibitem{LQZ} Lu, Peng; Qing, Jie; Zheng, Yu, \emph{A note on conformal Ricci flow}. ArXiv:1009.5377

\bibitem{Ra} Rauch, Jeffrey, \emph{Partial Differential Equations}. Springer-Verlag, New York, (1991)

\end{thebibliography}
\end{document}